\documentclass[12pt]{amsart}
\usepackage{bbm}
\usepackage{tikz,tikz-3dplot}
\usepackage{tikz-cd}
\usepackage{todonotes}
\usetikzlibrary{matrix,arrows,decorations.pathmorphing,arrows.meta}
\usepackage{MnSymbol}
\usepackage{calligra,mathrsfs}
\usepackage{latexsym}
\usepackage{pgf}
\usepackage{mathtools}
\usepackage{stmaryrd}
\usepackage{atbegshi}
\usepackage{algpseudocode}
\usepackage{algorithm}
\usepackage{adjustbox}
\usepackage{enumitem}
\usepackage[nocompress]{cite}

\usepackage[margin=1.5in]{geometry}
\usepackage{pictex}
\usepackage{amsmath,amscd}
\usepackage{afterpage}
\usepackage{youngtab}
\usepackage{psfrag}
\usepackage{ytableau}
\usepackage{verbatim}
\usepackage{graphics}
\usepackage{changepage}

\newtheorem{theorem}{Theorem}[section]
\newtheorem{proposition}[theorem]{Proposition} 
\newtheorem{definition}[theorem]{Definition}
\newtheorem{lemma}[theorem]{Lemma}
\newtheorem{corollary}[theorem]{Corollary} 
\newtheorem{example}[theorem]{Example}
\newtheorem{remark}[theorem]{Remark}

\newtheorem{maintheorem}{Theorem}

\makeatletter
\newcommand*{\da@rightarrow}{\mathchar"0\hexnumber@\symAMSa 4B }
\newcommand*{\da@leftarrow}{\mathchar"0\hexnumber@\symAMSa 4C }
\newcommand*{\xdashrightarrow}[2][]{%
  \mathrel{%
    \mathpalette{\da@xarrow{#1}{#2}{}\da@rightarrow{\,}{}}{}%
  }%
}
\newcommand{\xdashleftarrow}[2][]{%
  \mathrel{%
    \mathpalette{\da@xarrow{#1}{#2}\da@leftarrow{}{}{\,}}{}%
  }%
}
\newcommand*{\da@xarrow}[7]{%
  \sbox0{$\ifx#7\scriptstyle\scriptscriptstyle\else\scriptstyle\fi#5#1#6\m@th$}%
  \sbox2{$\ifx#7\scriptstyle\scriptscriptstyle\else\scriptstyle\fi#5#2#6\m@th$}%
  \sbox4{$#7\dabar@\m@th$}%
  \dimen@=\wd0 %
  \ifdim\wd2 >\dimen@
    \dimen@=\wd2 %   
  \fi
  \count@=2 %
  \def\da@bars{\dabar@\dabar@}%
  \@whiledim\count@\wd4<\dimen@\do{%
    \advance\count@\@ne
    \expandafter\def\expandafter\da@bars\expandafter{%
      \da@bars
      \dabar@ 
    }%
  }%  
  \mathrel{#3}%
  \mathrel{%   
    \mathop{\da@bars}\limits
    \ifx\\#1\\%
    \else
      _{\copy0}%
    \fi
    \ifx\\#2\\%
    \else
      ^{\copy2}%
    \fi
  }%   
  \mathrel{#4}%
}
\makeatother

\newcommand{\RR}{\mathbb{R}}
\newcommand{\CC}{\mathbb{C}}
\newcommand{\ZZ}{\mathbb{Z}}
\newcommand{\NN}{\mathbb{N}}

\newcommand{\Perm}{\operatorname{Perm}}

\newcommand*{\sheafhom}{\mathscr{H}\kern -.5pt om}

\newcommand{\BP}{\operatorname{BP}}

\newcommand{\PS}{\operatorname{PS}}

\newcommand{\MV}{MV }

\title{Mirkovi\'c-Vilonen Polytopes From Combinatorics}

\author{Mario Sanchez}
\address{Cornell University}
\email{ms2962@cornell.edu}
\usepackage{hyperref}
\hypersetup{
    colorlinks,
}

\begin{document}
 \begin{abstract}

Mirkovi\'c-Vilonen (MV) polytopes are a class of generalized permutahedra originating from geometric representation theory. In this paper we study \MV polytopes coming from matroid polytopes, flag matroid polytopes, Bruhat interval polytopes, and Schubitopes. We give classifications and combinatorial conditions for when these polytopes are \MV polytopes. We also describe how the crystal structure on \MV polytopes manifests combinatorially in these situations. As a special case, we show that the Newton polytopes of Schubert polynomials and key polynomials are \MV polytopes.
\end{abstract}
\maketitle

\tableofcontents
%%%%%%%%%%%%%%%%%%%%%%%%%%%%%%%%%%%%%%%%%%%%%%%%%%
% Introduction
%%%%%%%%%%%%%%%%%%%%%%%%%%%%%%%%%%%%%%%%%%%%%%%%%%

\section{Introduction}

Mirkovi\'c-Vilonen (MV) polytopes are a beautiful class of generalized permutahedra, introduced by Anderson \cite{Anderson2003}, which orginiated from geometric representation theory as the moment map images of Mirkovi\'c-Vilonen cycles in the affine Grassmannian. Since their inception, they have been connected to many other fields including the representation theory of preprojective algebras \cite{Baumann2012}, cluster algebras \cite{Anderson2006}, Khovanov-Lauda-Rouquier algebras \cite{Tingley2016}, quiver theory \cite{Saito2012}, and crystals \cite{Kamnitzer2007,Jiang2017,Naito2008}.

In \cite{Kamnitzer2007}, Kamnitzer gave a combinatorial characterization of \MV polytopes using the tropical Pl\"ucker relations from Berenstein and Zelevinsky's work on total positivity \cite{Berenstein1997}. Despite this combinatorial description and deep connections to other fields, these polytopes have not yet found their way into the combinatorics community. The goal of this paper is to remedy this. We study the intersection of various natural combinatorial families of generalized permutahedra with the set of MV polytopes.

In Section \ref{Sec: Background}, we introduce the background on \MV polytopes. In this paper, we focus only on the type $A_{n}$ case, and so, we adapt our definitions to that case. In Section \ref{Sec: First Examples}, we discuss a few first examples coming from graphic zonotopes, graph associahedra, and Pitman-Stanley polytopes.

In Section \ref{Sec: Matroids}, we study \MV polytopes coming from  matroids. We show that the matroids with \MV matroid polytopes are exactly the lattice path matroids first studied by Bonin, de Mier, and Noy \cite{Bonin2003}.  Combinatorially, these are the matroids whose bases are the set of east-north lattice paths contained in a given skew-diagram. Geometrically, these are the matroids whose matroid polytopes are the moment map images of the Richardson varieties in the Grassmannian.

Our main approach is to describe the raising and lowering operators from the crystal on \MV polytopes combinatorially in terms of the bases of the matroid and then use these to show that every lattice path matroid polytope can be obtained from a point by applying a sequence of raising operators. With this, we obtain our first main theorem.

\begin{maintheorem}\label{mainthm: MV polytopes lattice path matroids}
    The base polytope of a matroid $M$ is an MV polytope if and only if $M$ is a lattice path matroid.
\end{maintheorem}

In Section \ref{Sec: Flag Matroids and BIPS}, we study a generalization of matroids known as flag matroids. These flag matroids consist of a collection of matroids $(M_1,\ldots,M_k)$ such that $M_i \twoheadrightarrow M_{i-1}$ is a matroid quotient for all $i = 2, \ldots, k$. We give a classification of which flag matroids have an \MV flag matroid polytope.

\begin{maintheorem}\label{mainthm: MV polytopes flag matroids}
    The flag matroid polytope of a flag matroid $\mathcal{M} = (M_1,\ldots, M_k)$ is an MV polytope if and only if $P(M_i)$ is an \MV polytope for each matroid.
\end{maintheorem}
By Theorem \ref{mainthm: MV polytopes lattice path matroids}, this happens exactly when each $M_i$ is a lattice path matroid. Such flag matroids are called lattice path flag matroids and were recently studied by Bendetti and Knauer \cite{benedettiLPFM}.

A special class of flag matroids are the Bruhat interval polytopes introduced by Kodama and Williams \cite{Kodama2014}. For an interval $[u,v]$ in the Bruhat order of $S_n$, the \textbf{(twisted) Bruhat interval polytope} $\tilde{P}_{[u,v]}$ is the polytope
 \[\tilde{P}_{[u,v]} = \operatorname{conv}((n+1-w^{-1}(1), n+1-w^{-1}(2), \ldots, n+1-w^{-1}(n)) \; \lvert \; w \in [u,v]). \]
 Our characterization of \MV flag matroid polytopes also characterizes the \MV Bruhat interval polytopes. For any $k$, let $\operatorname{proj}_k: S_n \to \binom{[n]}{k}$ be the map
  \[\operatorname{proj}_k(w) = \{w(1), \ldots, w(k)\}. \]
 These are the projection maps onto the maximal parabolic quotients of $S_n$.
\begin{corollary}
    A twisted Bruhat interval polytope $\tilde{P}_{[u,v]}$ is an \MV polytope if and only if $\operatorname{proj}_k([u,v]) = [\operatorname{proj}_k(u),\operatorname{proj}_k(v)]$ for all $k \in [n]$.
\end{corollary}
We do not know a general description of which intervals satisfy this projection condition. However, we show that this condition holds for intervals $[u,v]$ where $u \leq v$ in the weak Bruhat order.

In the Section \ref{Sec: Schubitope}, we move to our final class of examples. In \cite{Magyar1998}, Magyar constructed a collection of polynomials coming from the characters of certain representations of the Borel subgroup $B \subseteq \operatorname{GL}_n$ coming from Bott-Samelson varieties. These representations are indexed by multisets of subsets $\mathcal{D} = \{C_1,\ldots,C_k\}$ of $[n]$ called \textbf{diagrams}. We say that a diagram is \textbf{strongly separated} if for any two $C_1, C_2 \in \mathcal{D}$ we have
 \[(C_1 \backslash C_2) \leq_{elt} (C_2 \backslash C_1) \quad \quad \text{or} \quad \quad (C_2 \backslash C_1) \leq_{elt} (C_1 \backslash C_2),\]
where $A \leq_{elt}B$ if every element of $A$ is smaller than every element of $B$.

The main motivation for these polynomials is that they include Schubert polynomials and key polynomials. Further, the action of the Demazure isobaric divided difference operator on these characters has a nice combinatorial description on the corresponding diagrams.

The Newton polytopes of these characters are called \textbf{Schubitopes}. It was proven by Fink, M\'ez\'aros, and St. Dizier that every Schubitope is a generalized permutahedron \cite{Fink2018}. In fact, they are Minkowski sum of matroid polytopes of Schubert matroids.

As with lattice path matroids, we give a combinatorial interpretation of the raising and lowering operators of the crystal on MV polytopes. We show that applying the operator $\mathbf{e_i}$ to a Schubitope corresponds to applying the Demazure operator to the underlying character. With this, we obtain a large class of \MV Schubitopes.

\begin{maintheorem}\label{mainthm: MV polytopes Schubitopes}
    The Schubitope $\mathcal{S}_\mathcal{D}$ is an MV polytope whenever $\mathcal{D}$ is strongly separated.
\end{maintheorem}

As a consequence, we obtain the following relationship between symmetric functions and \MV polytopes.
\begin{corollary}
    The Newton polytopes of Schubert polynomials and key polynomials are MV polytopes.
\end{corollary}

\section*{Acknowledgments}

Many thanks to Allen Knutson for teaching me about the affine Grassmannian and suggesting the combinatorial study of Mirkovi\'c-Vilonen polytopes. Thanks to Melissa Sherman-Bennett for helpful chats about Bruhat interval polytopes. The author was partially funded by the NSF Postdoctoral Research Fellowship DMS-2103136

\section{Background}\label{Sec: Background}

In this paper, we will primarily focus on the type $A_n$ case leaving the generalizations to arbitrary types for a later paper. The main reason is that some of the constructions and properties we rely on are only understood in this case.

We use the notation $[n] = \{1,\ldots, n\}$. For a finite set $E$, let $\RR^E$ be the vector space of real-valued functions on $E$. We typically view this through the standard basis $\{e_i\}_{i \in E}$. Let $e_S = \sum_{i \in S} e_i$ for any $S \subseteq E$.

\subsection{Generalized Permutahedra}

MV polytopes are subfamilies of a class of polytopes which have recently received much attention in algebric combinatorics. These polytopes were first introduced by Edmonds \cite{Edmonds2003} in the field of optimization as polymatroids. They were later studied geometrically by Kamnitzer \cite{Kamnitzer2007} as pseudo-Weyl polytopes and combinatorially by Postnikov as generalized permutahedra \cite{Postnikov2009}. We will use the latter name.

\begin{definition}
A \textbf{generalized permutahedron} on ground set $E$ is a polytope in $\RR^E$ such that every edge direction is parallel to a vector of the form $e_i - e_j$.
\end{definition}

In other words, a polytope $P$ is a generalized permutahedron if and only if its normal fan coarsens the fan given by the type $A_{n-1}$ Coxeter arrangement, otherwise known as the braid arrangement, which is defined as the collection of hyperplanes
 \[H_{i,j} = \{x \in \RR^E \; \lvert \; x_i \not = x_j \}, \]
for all $i \not = j \in E$.

We will parameterize generalized permutahedra by two dual types of functions. A \textbf{submodular function} on ground set $E$ is a monotonic function $\mu: 2^{E} \to \RR$ with $\mu(\emptyset) = 0$ satisfying the condition 
 \[\mu(S \cap T) + \mu(S \cup T) \leq \mu(S) + \mu(T).\]
Dually, a \textbf{supermodular function} on ground set $E$ is a monotonoic function $\mu: 2^{E} \to \RR^E$ with $\mu(\emptyset) = 0$ satisfying the condition 
 \[ \mu(S \cap T) + \mu(S \cup T) \geq \mu(S) + \mu(T).\]
The following is a standard bijection:

\begin{theorem}\cite{Fujishige2013,Aguiar2017}
    There is a bijection between submodular functions on ground set $E$ and generalized permutahedra on ground set $E$ by $P \mapsto \mu_P$ defined by \[\mu_P(S) = \max_{x \in P}(e_S \cdot x)\]
    and
     \[\mu \mapsto P_{\mu} = \{x \in \RR^E \; \lvert \; \sum_{i \in E}x_i = \mu(E), \sum_{i \in S} x_i \leq \mu(S) \text{ for all $S \subset E$}\}.  \]
\end{theorem}
There is a similar bijection between generalized permutahedra and supermodular functions where $P \mapsto \mu^P$ is defined by $\mu^P(S) = \min_{x \in P}(e_S \cdot x)$. The supermodular function of $P$ is also called the \textbf{Berenstein-Zelevinsky} data of $P$ in \cite{Kamnitzer2007}. These two functions are related by $\mu^P(S) = \mu_P([n]) - \mu_P([n]\backslash S)$.

Since the maximal dimensional cones of the braid arrangement are in bijection with $S_n$, we have a surjective map from $S_n$ to the vertices of a generalized permutahedron $P$. We use $v_w$ to denote the vertex of $P$ corresponding to $w \in S_n$. It is explicitly given as
 \[ v_w = \min_{v \in \operatorname{Vert}(P)}((w^{-1}(1),w^{-1}(2),\ldots,w^{-1}(n)) \cdot v).\]
The fact that there is a unique minimum follows from the definition of generalized permutahedra. This function is referred to as the \textbf{GGMS data} of $P$ in \cite{Kamnitzer2007}. We call $v_e$ the \textbf{lowest coweight} of $P$ and $v_{w_0}$ the \textbf{highest coweight} of $P$.

We say that a generalized permutahedron is a \textbf{lattice generalized permutahedron} if every vertex is in $\ZZ^E$. Equivalently, this happens when $\mu_P(S)$ and $\mu^P(S)$ are integers for all subsets $S \subseteq E$.

Every generalized permutahedron has codimension at least $1$. Indeed, a generalized permutahedron $P$ on ground set $E$ with submodular function $\mu$ is contained in the hyperplane defined by the equation $\sum x_i = \mu(E)$. We refer to $\mu(E)$ as the \textbf{rank} of $P$.

Recall that the Minkowski sum of two polytopes $P_1$ and $P_2$ is the polytope
 \[P_1 + P_2 = \{ x + y \; \lvert \; x \in P_1, y \in P_2 \}.\]
The Minkowski sum of two generalized permutahedra is again a generalized permutahedra and further
 \[\mu_{P+Q} = \mu_{P} + \mu_{Q}. \]
\subsection{Mirkovi\'c-Vilonen Polytopes}

Since we will not be focusing on the geometry in this paper, we will use the combinatorial description of MV polytopes due to Kamnitzer \cite{Kamnitzer2007}. We will only be using this in type $A$ so we restrict his characterization to this case. The following is a restatement of his result:

\begin{theorem}\cite{Kamnitzer2007}
    Let $E$ be a finite set with a total order $\leq$. Let $S\subseteq E$ and let $a,b,c \in E \backslash S$ such that $a < b < c$. Then, a lattice generalized permutahedron $P$ with supermodular function $\mu^P$ satisfies the $(S,a,b,c)$ \textbf{positive tropical Pl\"ucker relation} if
     \[\mu^P(Sb) + \mu^P(Sac) = \min \left( \mu^P(Sa) + \mu^P(Sbc), \mu^P(Sab) + \mu^P(Sc) \right ). \]
    Then $P$ is a \textbf{Mirkovi\'c-Vilonen (MV) polytope} on ground set $E$ if it satisfies the $(S,a,b,c)$ positive tropical Pl\"ucker relation for all tuples $(S,a,b,c)$.
\end{theorem}

We will also use the submodular version of the tropical Pl\"ucker relations. Following the translation between supermodular and submodular functions, we have that $P$ satisfies the submodular version of the tropical Pl\"ucker relation $(S,a,b,c)$ for $S \subseteq E$ and $a < b < c \in [n] \backslash S$ if
 \[\mu_P(Sac) + \mu_P(Sb) = \max(\mu_P(Sbc) + \mu_P(Sa), \mu_P(Sab) + \mu_P(Sc)). \]

The set of \MV polytopes is closed under translations. A \textbf{stable Mirkovi\'c-Vilonen polytope} is an equivalence class of an \MV polytopes up to translation. Let $\mathcal{MV}_{\lambda}$ be the set of \MV polytopes with lowest coweight $\lambda \in \ZZ^E$. For any $\lambda$, this gives a choice of representatives for stable \MV polytopes.

Kamnitzer proved the a useful characterization of \MV polytopes in terms of its two-dimensional faces. While we will not be using it in this paper, the reader might find it helpful.

\begin{proposition}\cite{Kamnitzer2007}
    Let $E$ be a finite set with a total order $\leq n$. Let $P$ be a generalized permutahedron in $\RR^E$ of dimension $\geq 2$. Then, $P$ is an \MV polytope if and only if every two-dimensional face is an \MV polytope in $\RR^E$.
\end{proposition}

\begin{remark}
    Being an \MV polytope is far from a combinatorial invariant of the polytope. The condition depends heavily on the embedding of the polytope in $\RR^E$. For instance, it is true that every generalized permutahedron in $\RR^2$ is an \MV polytope. These polytopes are all line segments or points since they must have codimension at least $1$. However, the generalized permutahedron which is the line segment $[e_1,e_3]$ in $\RR^3$ is not an \MV polytope. 
\end{remark}

MV polytopes are closed under products as long as one is careful with the total orders.

\begin{proposition}
    Let $E$ be a finite set with total order $\leq$. Let $S_1 \sqcup \cdots \sqcup S_k = E$ with the property that every element of $S_i$ is larger than every element of $S_{i-1}$. Let $P_1, \ldots, P_k$ be \MV polytopes in $\RR^{S_1}, \ldots, \RR^{S_k}$, respectively. Then, $P_1 \times \cdots \times P_k \subseteq \RR^E$ is an \MV polytope on ground set $E$.
\end{proposition}

\subsection{The Crystal of MV Polytopes}

Let $\Phi$ be a root system with weight lattice $\Lambda$ and simple roots $\alpha_1,\ldots,\alpha_n$. For our purposes, a \textbf{crystal} is a set $\mathcal{C}$ together with the structure maps
 \[ \mathbf{e}_i: \mathcal{C} \to \mathcal{C}, \quad \quad \mathbf{f}_i: \mathcal{C} \to \mathcal{C} \sqcup \{0 \}, \quad \quad \mathbf{wt}: \mathcal{C} \to \Lambda,\]
which satisfy the axioms

\begin{enumerate}
    \item For $a \in \mathcal{C}$ and $\mathbf{e}_i$ with $\mathbf{e}_i \cdot a \not = 0$, then $\mathbf{wt}(\mathbf{e}_i \cdot a) = \mathbf{wt}(a) + \alpha_i$.
    \item For $a \in \mathcal{C}$ and $\mathbf{f}_i$ with $\mathbf{f}_i \cdot a \not = 0$, then $\mathbf{wt}(\mathbf{f}_i \cdot a) = \mathbf{wt}(a) - \alpha_i$.
    \item $a' = \mathbf{e}_i \cdot a$ if and only if $ \mathbf{f}_i \cdot a' = a$.
\end{enumerate}

We refer to $\mathbf{e}_i$ as the raising operator, $\mathbf{f}_i$ as the lowering operator, and $\mathbf{wt}$ as the weight function.

Each representation of a Lie algebra $\mathfrak{g}$ gives rise to a natural crystal of the corresponding root system. A particularly notable example is the crystal $B(\infty)$ corresponding the the Verma module of $\mathfrak{g}$ whose underlying set is the Kashiwara crystal basis, which is the same as Lusztig's canonical basis, for $U_{-}^{\vee}$. Braverman, Finkelberg, and Gaitsgory gave a geometric description of this crystal where the underlying set is the set of stable \MV cycles. Both of these crystals induce a crystal on the set of stable \MV polytope. In \cite{Kamnitzer2010}, Kamnitzer showed that these two crystals coincide.

We describe this crystal on stable \MV polytopes in type $A_{n-1}$. We choose (slightly nonstandard) conventions so that our weight lattice is $\ZZ^n$ and our simple roots are the vectors $\alpha_i = e_{i+1} - e_i$ for $i=1,\ldots,n-1$.

\begin{theorem}\cite{Kamnitzer2010}
    For any integer vector $\lambda \in \ZZ^n$, there is a crystal on $\mathcal{MV}_{\lambda}$ which has weight function $\mathbf{wt}(P) = v_{w_0}$ and has raising and lowering operators defined on an \MV polytope with vertices $\{v_w\}_{w \in S_n}$ as follows:
    
    The polytope $\mathbf{e}_i \cdot P$ is the unique \MV polytope with vertices $v_w'$ such that $v_{w_0}' = v_{w_0} + (e_{i+1} - e_i)$ and $v_w' = v_w$ for all $w$ with $w \leq w s_i$.

    The lowering operator satisfies $\mathbf{f}_i \cdot P = 0$ if and only if $v_e = v_{s_i}$. Otherwise, it is the unique \MV polytope with vertices $v_w'$ such that $v_{w_0}' = v_{w_0} - (e_{i+1} - e_i)$ and $v_w' = v_w$ for all $w$ with $w \leq w s_i$.
\end{theorem}

From this description, one can calculate the supermodular function of $\mathbf{e}_i \cdot P$ and $\mathbf{f}_i \cdot P$ by recursively solving a collection of $(\min, +)$ equations. However, this approach is not useful for our combinatorial needs. Instead we will use a more explicit description conjectured by Anderson and Mirkovi\'c and proven by Kamnitzer in the type $A$ case.

\begin{theorem}\label{prop: supermodular function of crystal action}\cite{Kamnitzer2010}.
    Let $P$ be an \MV polytope of type $A_{n-1}$. Let $\mu$ be the supermodular function of $P$ and $\mu'$ be the supermodular function of $\mathbf{e}_i \cdot P$. Then,

    \[\mu'(S) = \begin{cases}
        \min(\mu(S), \mu(S \backslash i \cup \{i+1\}) + c) & \text{if $i \in S$ and $i+1 \not \in S$}\\
        \mu(S) & \text{otherwise.} \\
    \end{cases}, \]
\end{theorem}

We can rephrase this in its submodular version.

\begin{proposition}\label{prop: submodular function of crystal action}
    Let $P$ be an \MV polytope of type $A_{n-1}$. Let $\mu$ be the submodular function of $P$ and $\mu'$ be the submodular function of $\mathbf{e}_i \cdot P$. Then,
     \[\mu'(S) = \begin{cases}
        \max( \mu(S), \mu(S \backslash \{i+1\} \cup \{i\}) - c) & \text{if $i \not \in S$ and $i+1 \in S$} \\
         \mu(S) & \text{otherwise,} \\
     \end{cases} \]
    where $c = \mu([n] \backslash s_i[i]) - \mu([n] \backslash [i]) - 1$.
\end{proposition}

\section{First Examples}\label{Sec: First Examples}

We begin with simple examples to demonstrate the basics of the theory. These examples can be studied by directly computing the submodular function and checking the tropical Pl\"ucker relations. This approach will not work for our later sections.

\subsection{Graphic Zonotopes}

The simplest class of generalized permutahedra comes from the theory of graphs. Given a simple graph $G$ on ground set $[n]$ with the usual order, the \textbf{graphic zonotope} $Z_G$ of the graph $G$ is the Minkowski sum
 \[ Z_G = \sum_{e = \{i,j\} \in E(G)} [e_i, e_j]. \]

\begin{lemma}\label{lem: supermodular function of graphic zonoteops}
    Let $\mu$ be the supermodular function of $Z_G$, then
     \[\mu(S) = |E(G|_S)|, \]
    where $G|_S$ is the induced subgraph on vertex set $S$.
\end{lemma}

\begin{proof}
    The supermodular function of a Minkowski sum is the sum of the supermodular functions of the summands. In this case, the minimal value of $e_S \cdot x$ for $x \in [e_i,e_j]$ is $1$ if both $i$ and $j$ are in $S$ and $0$ otherwise.
\end{proof}

\begin{example}
    The most famous example is the graphic zonotope of the complete graph $K_n$ which is the permutahedron
     \[Z_{K_n} = \operatorname{Perm}_n = \operatorname{conv}(w \cdot (0,\ldots,n-1) \; \lvert \; w \in S_n). \]
    This is an \MV polytope. More generally, the polytopes
     \[\operatorname{Perm}(\lambda) = \operatorname{conv}(w \cdot \lambda \; \lvert \; w \in S_n) \]
    are \MV polytopes. Indeed, it is easy to show that the submodular function satisfies $\mu(S) = \lambda_1 + \lambda_2 + \cdots + \lambda_{|S|}$ where $\lambda_i$ is the $i$th largest coordinate of $\lambda$. From this, the tropical Pl\"ucker conditions trivially hold.

    This also immediately follows from the geometry since $\Perm(\lambda)$ is the moment map image of the spherical Schubert variety $\Omega_{\lambda} \subseteq \operatorname{Gr}$ which is an \MV cycle.
\end{example}

We now show that the complete graph is essentially the only example of an \MV graphic zonotope.

\begin{proposition}
    The graphic zonotope $Z_G$ of a simple graph $G$ on ground set $[n]$ is an MV polytope if and only if $G$ is a product of complete graphs and the blocks in the set partition of $[n]$ induced by the connected components of $G$ consists of intervals of $[n]$.
\end{proposition}

\begin{proof}
    Since \MV polytopes are closed under polytope products when the ground sets are intervals, we immediately obtain the if direction.

    For the only if direction, let $a,b,c \in [n]$ be integers with $a < b < c$ such that $(a,c) \in E$ but one of $(a,b)$ or $(b,c)$ is not in $E(G)$. Such a choice exists exactly when $G$ does not satisfy the conditions of the proposition. Then, since $\mu(\{i,j\})$ is $1$ if $(i,j) \in E(G)$ and $0$ otherwise, we have that
     \[\mu(b) + \mu(ac) = 1,\]
    while
    \[\min(\mu(a) + \mu(bc), \mu(ab) + \mu(c)) = 0. \]
\end{proof}

\subsection{Graph Associahedra}

Given a simple graph $G$ on ground set $E$, a \textbf{tube} $\tau$ of $G$ is a subset of vertices of cardinality\footnote{This is different than the usual definition but it only changes the polytope by a translation which does not affect whether it is an \MV polytope.} greater than $1$ such that the induced subgraph on those vertices is connected. Let $\operatorname{tubes}(G)$ be the set of tubes of $G$. The \textbf{graph associahedron} $A_G$ of $G$ is the Minkowski sum
 \[A_G = \sum_{\tau \in \operatorname{tubes}(G)} \Delta_{\tau}, \]
where $\Delta_{\tau} = \operatorname{conv}(e_i \; \lvert \; i \in \tau)$. These objects were first studied as generalizations of the associahedron in \cite{Carr2006}.

By the same arguments as in the previous sections, we have following description of the supermodular function:

\begin{lemma}
    Let $G$ be a simple graph on ground set $E$ and $\mu$ be the supermodular function of $A_G$. Then,
     \[\mu(S) = | \{ \tau \in \operatorname{tubes}(G) \; \lvert \; \tau \subseteq S\} |. \] 
\end{lemma}

The class of graph associahedra include many well-studied polytopes.

\begin{enumerate}
    \item When $G$ is the complete graph $K_n$, the graph associahedron $A_G$ is the permutahedron.
    \item When $G$ is the path graph on $n$ vertices, the graph associahedron is Loday's realization of the associahedron \cite{Loday2004}.
    \item When $G$ is the star graph on $n$ vertices, $A_G$ is the stellahedron studied which plays an important role in the augmented Hodge theory of matroids, see for instance \cite{EHL22}.
    \item When $G$ is the cycle graph on $n$ vertices, $A_G$ is the Bott-Taubes cyclohedron which appears in knot theory \cite{Bott1994}.
\end{enumerate}

We classify which graph associahedra are \MV polytopes. Again, the only example is essentially the permutahedron. Since the graph associahedron of a disconnected graph is the product of the graph associahedra of the connected components, it suffices to consider the connected case.

\begin{proposition}
    Let $G$ be a connected graph on ground set $E$ with a total order $\leq$. Then, $A_G$ is an \MV polytope if and only if $G$ is the complete graph.
\end{proposition}

\begin{proof}
    We saw above that the complete graph gives the permutahedron which is an \MV polytope.

    In the other direction, since $G$ is connected but not the complete graph, there exist vertices $a < b < c$ such that $G|_{a,b,c}$ is a path graph on three vertices.

    Consider the positive tropical Pl\"ucker relation where $(S,a,b,c)$ where $S = \emptyset$. Let $\mu$ be the supermodular function of $\mathcal{N}_G$. Then $\mu(ij) = 1$ if  $\{i,j\}$ is an edge of $G$ and $0$ otherwise. Also, $\mu(i) = 0$ for all $i \in [n]$. Therefore, exactly two of the integers $\mu(b) + \mu(ac), \mu(a) + \mu(bc)$ and $\mu(ab) + \mu(c)$ are $1$ and the other integer is $0$. Therefore, 
     \[\mu(b) + \mu(ac) \not = \min(\mu(a) + \mu(bc), \mu(ab) + \mu(c)). \]
\end{proof}

\begin{corollary}
    The associahedron for $n \geq 3$, stellahedron for $n \geq 3$, and the cyclohedron for $n \geq 4$ are all not \MV polytopes for any choice of total order on the ground set.
\end{corollary}

\subsection{Pitman-Stanley Polytopes}

A \textbf{building set} on ground set $E$ is a collection of subsets $\mathcal{B} \subseteq 2^E$ such that if $S$ and $T$ are both in $\mathcal{B}$ and $S \cap T \not = \emptyset$, then $S \cup T \in \mathcal{B}$. The \textbf{nestohedron} $\mathcal{N}_{\mathcal{B}}$ is the polytope
 \[ \mathcal{N}_{\mathcal{B}} = \sum_{B \in \mathcal{B}} \Delta_B.\]

Every graph associahedron is an example of a nestohedron corresponding to the building set given by the tubes of $G$. Unlike graph associahedra, this class of polytopes includes examples which are not just products of permutahedra.

In \cite{Stanley2002}, Pitman and Stanley introduced polytopes whose combinatorics is closely related to the associahedron and the theory of parking functions. Given any $a_1, \ldots, a_n \in \ZZ$ with $a_i \geq 0$, the \textbf{Pitman-Stanley polytope} $\PS(a_1, \ldots, a_n)$ is the Minkowski sum
 \[\PS(a_1, \ldots, a_n) = \sum_{i=1}^{n} a_k \Delta_{[k]}, \]
where $\Delta_{[k]}$ is the simplex
\[\Delta_{[k]} = \operatorname{conv}(e_i \; \lvert \; i \in [k] ).\]
The Pitman-Stanley polytope $\PS(1,1,\ldots,1)$ is the nestohedron corresponding to the building set $\{ [k] \; \lvert \; k = 1,\ldots, n\}$.

\begin{lemma}\label{lem: supermodular function of Pitman-Stanley}
    Let $S \subseteq [n]$ with $m = \min(S)$ and $a_1, \ldots, a_n$ be non-negative integers. Let $\mu$ be the submodular function of $\PS(a_1,\ldots,a_n)$. Then,
     \[ \mu(S) = \sum_{i=m}^{n} a_i.\]
\end{lemma}

\begin{proof}
    As in the case of graphs, it suffices to sum the submodular functions of the simplices. Since the maximal value of $e_S \cdot x$ for $x \in \Delta_{[k]}$ is $1$ if $S \cap [k] \not = \emptyset$ and $0$ otherwise, we immediately get this result.
\end{proof}

\begin{proposition}
    The Pitman-Stanley polytope $P(a_1,\ldots,a_n)$ is an \MV polytope.
\end{proposition}

\begin{proof}

Let $S \subseteq [n]$ and $a < b < c \in [n] \backslash S$. We need to show that
 \[\mu(Sb) + \mu(Sac) = \max(\mu(Sa) + \mu(Sbc), \mu(Sab) + \mu(Sc)). \]

By Lemma \ref{lem: supermodular function of Pitman-Stanley}, we have that $\mu(T) \geq \mu(T')$ whenever $\min(T) \leq \min(T')$. Since $a < b < c$, we have $\min(Sa) = \min(Sc)$ and $\min(Sbc) \leq \min(Sc)$.
this implies
 \[\mu(Sa) + \mu(Sbc) \geq \mu(Sab) + \mu(Sc). \]
Further, $\min(Sb) = \min(Sbc)$ and $\min(Sac) = \min(Sa)$ which implies that
 \[\mu(Sb) + \mu(Sac) = \mu(Sa) + \mu(Sbc). \]
\end{proof}

We leave a detailed study of which nestohedra are \MV polytopes for a different paper.

\section{Matroids}\label{Sec: Matroids}
We now turn to giving a complete classification of which matroid polytopes are MV polytopes. See \cite{Oxley} for a general reference on matroids. For a matroid $M$ on ground set $[n]$, let $\mathcal{B(M)}$ denote its set of bases. The \textbf{matroid polytope} $P(M)$ of $M$ is the polytope
 \[P(M) = \operatorname{conv}(e_B \in \RR^n \; \lvert \; B \in \mathcal{B}(M)). \]

This polytope is a generalized permutahedron and its submodular function is the rank function of the matroid $\mu$, which is given by
 \[ \mu(S) = \max_{B \in \mathcal{B}(M)} |S \cap B|.\]
Gelfand, Gorensky, MacPherson, and Serganova proved the following classification of matroid polytopes:

\begin{theorem}\label{thm: GGMS}\cite{gelfand1987combinatorial}
    A polytope $P \subseteq \RR^n$ is a matroid polytope of a matroid of rank $k$ if and only if $P$ is a lattice generalized permutahedron and every vertex is contained in the hypersimplex
     \[\Delta_{k,n} = \{x \in \RR^n \; \lvert \; 0 \leq x_i \leq 1, \sum_{i=1}^nx_i = k\}. \]
\end{theorem}

\subsection{Lattice Path Matroids}

The matroids with \MV matroid polytopes will come from the the class of lattice path matroids first studied by by Bonin, de Mier, and Noy \cite{Bonin2003}. These matroids have combinatorial models given in terms of lattice paths which simplify many arguments. For our purposes, it will suffice to understand their bases in terms of subsets.

The \textbf{Gale order} on the subsets of $[n]$ of cardinality $k$ is the poset given by $A \leq B$ where $A = \{a_1, \ldots, a_k\}$ and $B = \{b_1, \ldots, b_k\}$ whenever $a_i \leq b_i$ for all $i \in [k]$. We have the following quick observation:

\begin{lemma}\label{lem: slice description of gale order}

We have $A \leq B$ if and only if for all $i \in [n]$
 \[|A \cap \{i,\ldots, n\}| \leq |B \cap \{i, \ldots, n\}|. \]    
\end{lemma}

For $S,T \subseteq [n]$ of the same cardinality such that $S \leq T$, the \textbf{lattice path matroid} $M([S,T])$ is the matroid whose bases are the subsets in the interval $[S,T] = \{ R \in \binom{[n]}{k} \; \lvert \; S \leq R \leq T\}$. The \textbf{Schubert matroid} $\Omega_A$ is the lattice path matroid whose bases are all subsets of $[n]$ of cardinality $|A|$ which are less than $A$ in the Gale order.

By the greedy algorithm of matroids, it follows that the lowest coweight of $P(M[S,T])$ is $v_e = e_S$ and the highest coweight is $v_{w_0} = e_T$.

We will use the following lemma throughout:

\begin{lemma}\label{lemma: value of c for LPM}
    Let $M = M[S,T]$ be a lattice path matroid with rank function $\mu$ such that $i+1 \in T$ and $i \not \in T$ and $S \leq T \backslash\{i+1\} \cup \{i\}$. Then, $c = \mu([n] \backslash s_i[i]) - \mu([n] \backslash [i]) - 1$ is equal to $0$.
\end{lemma}

\begin{proof}
    We have that $\mu([n] \backslash [i]) = \max_{B \in \mathcal{B}(M)}(|B_{\geq i}|)$. By Lemma \ref{lem: slice description of gale order}, we have that the this is always maximized at subset which is largest in the Gale order, which is the subset $T$ by construction.

    Since $[n]\backslash s_i[i]$ and $[n]\backslash [i]$ differ by one element, we have that their ranks can differ by at most $1$. Then, considering $T \backslash \{i +1\} \cup \{i\}$ which is in $[S,T]$ by assumption, we see that $\mu([n] \backslash s_i[i]) > \mu([n] \backslash [i])$. This implies the wanted result.
\end{proof}

These matroids have a nice interpretation coming from algebraic geometry. Let $\operatorname{Gr}(k,n)$ be the Grassmannian of $k$-planes in $\CC$. We can represent each $k$-plane by a full-rank $k \times n$ matrix whose rowspace is the wanted $k$-plane. For such a matrix $M \in \operatorname{Gr}(k,n)$ and a subset $S \in \binom{[n]}{k}$, the \textbf{Pl\"ucker coordinate} $p_S$ is the minor $M_S$ obtained by restricting to the columns indexed by $S$. The \textbf{Schubert variety} $X^S$ is the closure of the set of $A \in \operatorname{Gr}(k,n)$ whose largest, under the Gale order, non-vanishing Pl\"ucker coordinate is $p_S$. The \textbf{Richardson variety} $X_S^T$ is the closure of the set of $A \in \operatorname{Gr}(k,n)$ whose largest non-vanishing Pl\"ucker coordinate is $T$ and smallest non-vanishing Pl\"ucker coordinate is $S$. Then, the matroid polytope of $\Omega_S$ is the moment map image of $X^S$ and the matroid polytope of $M[S,T]$ is the moment map image of $X_S^T$.

\subsection{Crystal Structure in Terms of Subsets}

Our main tool will be a combinatorial interpretation of the raising and lowering operators in terms of the subsets of the defining interval. As usual, the lowering operator is mostly an inverse to the raising operator so it suffices to study the latter. We use the notation $A_{\geq i} = A \cap \{i,\ldots,n\}$.

\begin{proposition}\label{prop: crystal of lattice path matroids}
    Let $M[A,B]$ be a lattice path matroid with $i \in B$ and $i+1 \not \in B$ such that $P(M[A,B])$ is an \MV polytope. Let $B'$ be the set $B \backslash \{i\}\cup \{i+1\}$. Then,
     \[ \mathbf{e}_i \cdot \BP(M[A,B]) = \BP(M[A,B'])\]
\end{proposition}

\begin{proof}

    First, we describe all the subsets in $[A,B']$. We of course have that $[A,B] \subseteq [A,B']$ since $B \leq B'$ in the Gale order. Let $C$ be a subset in the interval $[A,B']$ but not in the interval $[A,B]$. By construction, we have $|B'_{\geq k}| = |B_{\geq k}|$ for all $k \not = i+1$. Then, since $C \leq B'$ and $C \not \leq B$, we must have from Lemma \ref{lem: slice description of gale order} that
     \[|C_{\geq i+1}| > |B_{\geq i+1}|. \]

    Therefore, $i+1 \in C$ and $i \not \in C$. Further, $C \backslash \{i+1\} \cup \{i\} \in [A,B]$. Hence, the subsets in $[A, B']$ are subsets of the form $D$ where either $A \leq D \leq B$ or $D = C \backslash \{i\} \cup \{i+1\}$ with $A \leq C \leq B$ and $D \not \in [A,B]$. Let $I_1$ be the subsets of the first form and $I_2$ be the subsets of the second form.

    Now, let $\mu$ be the submodular function of $P([A,B])$ and $\mu'$ be the submodular function of $P([A,B'])$. We have that by the definition of the rank function,
     \[\mu'(S) = \max\left( \max_{D \in S}(|D \cap S|), \max_{D \in I_2} |D \cap S|) \right). \]
    In the case where $i \in S$ or $i+1 \not \in S$, we have that the both of the maximums above attain the same maximal value at some $C$ with $A \leq C \leq B$ by the definition of $I_2$. Therefore, $\mu'(S) = \mu(S)$.

    In the case where $i \not \in S$ and $i+1 \in S$, then this maximum is attained either at a subset $C \in [A,B]$ or a subset $D = C \backslash \{i\} \cup \{i+1\}$ with $C \in [A,B]$. In other words,
     \[\mu'(S) = \max(\mu(S), \mu(S \backslash \{i+1\} \cup \{i\})) \]
    
    We have shown that

    \[\mu_{P'}(S) = \begin{cases}
        \max(\mu(S), \mu(S \backslash \{i+1\} \cup \{i\})) & \text{if $i \not \in S$ and $i+1 \in S$.}\\
        \mu_{P}(S) & \text{otherwise}\\
    \end{cases} \]

    Comparing this to the description of the raising operator from Proposition \ref{prop: submodular function of crystal action}, it remains to show that  $c = \mu([n] \backslash s_i[i]) - \mu([n] \backslash [i]) - 1 = 0$. However, this follows from Lemma \ref{lemma: value of c for LPM}.
\end{proof}

Note that the elements $B'$ that cover $B$ are exactly the elements obtained by swapping $i$ with $i+1$ in $B$. Therefore, applying raising operators corresponds to moving the subset $B$ up in the Gale order by a single step in the poset.

\subsection{Classification of Mirkovi\'c-Vilonen Matroid Polytopes}

We now turn to the main result of this section which is a complete classification of the MV matroid polytopes. 

\begin{theorem}
    A matroid polytope $P(M)$ is an MV polytope if and only if $M$ is a lattice path matroid.
\end{theorem}

\begin{proof}
    First, we prove that if direction. Notice that the elements $B'$ obtained from $B$ by swapping $i$ and $i+1$ in $B$ as in Proposition \ref{prop: crystal of lattice path matroids} are include the elements $B'$ that cover $B$. Therefore, the matroid polytopes that can be obtained from $P(M[A,B])$ by using the raising operators are exactly the matroids $P(M[A,C])$ where $B \leq C$. Since every base polytope $P(M[A,A])$ is a point and hence is an MV polytope, this direction of the proof follows by choosing any maximal chain from $A$ to $B$ and applying the corresponding raising operators.

    For the only if direction, let $M$ be matroid on $[n]$ of rank $r$ such that $P(M)$ is an MV polytope. Let $A$ be the minimum element of $\mathcal{B}(M)$ in the Gale order. The greedy algorithm of matroids ensures that such a unique minimal element exists. The subset $A$ being smallest in the Gale order is equivalent to the vertex $e_A$ of $P(M)$ being the vertex minimized in direction $(1,2,3,\ldots,n)$. Hence, $e_A$ is the lowest coweight of $P(M)$. The polytope $P(M)$ can be obtained by applying a sequence of raising operators $\mathbf{e}_{i_1}\cdots \mathbf{e}_{i_\ell}$ to the polytope $\{e_A \}$, as the crystal $\mathcal{MV}_{e_A} \cong B(\infty)$ is connected and $\{e_A\}$ is the unique MV polytope that is killed by all lowering operators.

    Let $\{e_A\} = P_0, P_1,\ldots P_{\ell} = M$ be the sequence of lattice generalized permutahedra obtained by applying these operators in order. First, note that if $P_i$ is not a matroid polytope, then, by Theorem \ref{thm: GGMS} it is not contained in the hypersimplex $\Delta_{r,n}$. Equivalently, there is some subset $S \subseteq [n]$ such that $\mu_{P_i}(S) > \min(|S|, r)$. It is clear from the hyperplane description of the crystal operators in Proposition \ref{prop: submodular function of crystal action}, that for any generalized permutahedron $P$, we have $\mu_{\mathbf{e}_i \cdot P}(S) \geq \mu_{P}(S)$ for all $S \subseteq[n]$ and all $i \in [n]$. Therefore, if $P_i$ is not contained in the hypersimplex then neither is $P_{i+1}$. Since the operator $\mathbf{e}_i$ preserves the rank of the generalized permutahedron, this means that if $P_{i+1}$ is a matroid polytope then so is $P_{i}$. By induction, since $P_{\ell}$ is a matroid polytope, each $P_i$ is a matroid polytope.

    Finally, we show that each $P_i$ is a lattice path matroid by induction. We begin with $P_0 = \{e_A\} = P(M[A,A])$ which is a lattice path matroid. Then, suppose that $P(M[A,B])$ is a lattice path matroid and we are applying the operator $\mathbf{e}_i$. If $i \not \in B$, then $\mathbf{e}_i \cdot P(M[A,B])$ contains a vertex whose $i$th entry is $-1$ since the highest coweight of this polytope is $e_B + (e_{i+1} - e_i)$. Likewise, if $i+1 \in B$, then $\mathbf{e}_i \cdot P(M[A,B])$ contains a vertex whose $i+1$th entry is $2$. Since $\mathbf{e}_i \cdot P(M[A,B])$ is a matroid polytope, neither of these cases can happen. Therefore, $i \in B$ and $i+1 \not \in B$. Now Proposition \ref{prop: crystal of lattice path matroids} applies and so $\mathbf{e}_i \cdot P(M[A,B])$ is the lattice path matroid polytope $P(M[A,B \backslash \{i\} \cup \{i+1\}])$. By induction, this shows that $M$ is a lattice path matroid.
\end{proof}

The if direction of this result is also a consequence of the following result of Anderson and Kogan:

\begin{proposition}\cite{Anderson2006}\label{prop: Richardsons are MV cycles}
    The Richardson varieties in Grassmannians are MV cycles in the affine Grassmannian.
\end{proposition}

\section{Flag Matroids and Bruhat Interval Polytopes}\label{Sec: Flag Matroids and BIPS}
Matroid polytopes arise as moment map images of torus orbit closures in the Grassmannian. Generalizing this connection to flag varieties gives a generalization of matroids called flag matroids. 

Let $M$ and $N$ be matroids where the rank of $M$ is greater than the rank of $N$. We say that $N$ is a \textbf{matroid quotient} of $M$ denoted by $M \twoheadrightarrow N$ if for all $A \subseteq B$ we have
 \[\mu_M(B) - \mu_M(A) \geq \mu_N(B) - \mu_N(A) \]
 where $\mu_M$ and $\mu_n$ are the rank functions of $M$ and $N$, respectively. A \textbf{flag matroid} on ground set $[n]$ is a sequence of matroids $\mathcal{M} = (M_1,M_2,\ldots,M_k)$, also on ground sets $[n]$, such that 
 \[M_k \twoheadrightarrow M_{k-1} \twoheadrightarrow \cdots \twoheadrightarrow M_1. \]
We refer to the matroids $M_i$ as the \textbf{constituents} of $\mathcal{M}$. A set of matroids are \textbf{concordant} if they appear together as the constituents of a flag matroid. We say that $\mathcal{M}$ is a full flag matroid if $k = n$. The \textbf{flag matroid polytope} $P(\mathcal{M})$ is the polytope given as the Minkowski sum $P(M_1) + \cdots + P(M_k)$.

Following the work of Benedetti and Knauer \cite{benedettiLPFM}, we say that a flag matroid $\mathcal{M} = (M_1,\ldots, M_k)$ is a \textbf{lattice path flag matroid} if every component is a lattice path matroid.

\subsection{Quotients and Mirkovi\'c-Vilonen polytopes}

We now show that the flag matroids whose polytopes are \MV polytopes are exactly the lattice path flag matroids.

\begin{theorem}[Theorem \ref{mainthm: MV polytopes flag matroids}]
    Let $\mathcal{M} = (M_1,\ldots, M_k)$ be a flag matroid. Then $P(\mathcal{M})$ is an MV polytope if and only if each $P(M_i)$ is an MV polytope.

    In other words, the flag matroids which have MV flag matroid polytopes are exactly the lattice path flag matroids.
\end{theorem}

We split the proof into the two directions.

\begin{proposition}
    Let $\mathcal{M} = (M_1, \ldots, M_k)$ be a flag matroid where each $P(M_i)$ is an MV polytope. Then so is $P(\mathcal{M})$.
\end{proposition}

\begin{proof}
    Let $f_i$ be the supermodular function of $P(M_i)$ and $g$ be the supermodular function of $P(\mathcal{M})$. Let $S \subseteq [n]$ and $a,b,c \in [n] \backslash S$ with $a < b < c$. We need to show that $P(\mathcal{M})$ satisfies the $(S,a,b,c)$ positive tropical Pl\"ucker condition. Since $f_i$ comes from an MV polytope, we know that
     \[f_i(Sb) + f_i(Sac) = \min(f_i(Sa) + f_i(Sbc), f_i(Sab) + f_i(Sc)). \]
    We will show that the minimum is either always the left-side or the right-side for all $i=1,\ldots, k$.

    Let $x_i = f_i(Sbc)-f_i(Sc)$ and $y_i = f_i(Sab) - f_i(Sa)$. By the matroid quotient condition $M_{i+1} \twoheadrightarrow M_i$, we have that $x_i \geq x_{i+1}$ and $y_i \geq y_{i+1}$. Further, by the matroid condition we know that $x_i$ and $y_i$ are each either $0$ or $1$.

    If the pair $(x_{j}, y_{j}) = (0,0)$ for some $j$, then the matroid quotient condition implies that $(x_i, y_i) = (0,0)$ for all $i \geq j$. If $(x_{i},y_{i}) = (1,0)$ then the matroid quotient condition implies that the pair $(x_{i+1},y_{i+1})$ is either $(0,0)$ or $(0,1)$. Likewise, $(x_{i},y_{i}) = (0,1)$ implies that $(x_{i+1},y_{i+1})$ is either $(0,0)$ or $(0,1)$.

    All together, this implies that either $x_i \leq y_i$ for all $i=1,\ldots, k$ or $x_i \geq y_i$ for all $i = 1,\ldots, k$. Rearranging, this means that either
     \[  f_i(Sa) + f_i(Sbc) \leq f_i(Sab) + f_i(Sc)\]
    for all $i$ or the other way around.

    Therefore, we obtain

    \begin{align*}
        g(Sb) + g(Sac) &= \sum_{i=1}^k f_i(Sb) + f_i(Sac) \\
        &= \sum_{i=1}^k \min(f_i(Sa) + f_i(Sbc), f_i(Sab) + f_i(Sc)) \\
        &= \min\left( \sum_{i=1}^kf_i(Sa) + f_i(Sbc), \sum_{i=1}^k f_i(Sab) + f_i(Sc)\right) \\
        &= \min(g(Sa) + g(Sbc), g(Sab) + g(Sc)).
    \end{align*}

    Hence, $P(\mathcal{M})$ satisfies the $(S,a,b,c)$ positive tropical Pl\"ucker condition.
\end{proof}

We now prove the converse.

\begin{proposition}
    Let $\mathcal{M} = (M_1, \ldots, M_k)$ be a lattice path matroid such that $P(\mathcal{M})$ is an MV polytope. Then, every polytope $P(M_i)$ is an MV polytope.
\end{proposition}

\begin{proof}
    Let $g$ be the supermodular function of $P(\mathcal{M})$ and $f_i$ be the submodular function of $P(M_i)$. Let $S$ be a subset and $a,b,c \in [n] \backslash S$ such that $a < b < c$. We need to show that each $f_i$ satisfies the $(S,a,b,c)$ positive tropical Pl\"ucker relation. We know that $g(A) = \sum_{i=1}^{k} f_i(A)$ and that $g$ satisfies
     \[g(Sb) + g(Sac) = \min(g(Sa) + g(Sbc), g(Sab) + g(Sc)). \]

    We focus on the case where $g(Sb) + g(Sac) = g(Sa) + g(Sbc) \leq g(Sab) + g(Sc)$. The other case will follow by similar arguments. First, this gives
     \[\sum_{i=1}^{k} f_i(Sb) + f_i(Sac) = \sum_{i=1}^k f_i(Sa) + f_i(Sbc). \]
    Rearranging,
     \[ \sum_{i=1}^k f_i(Sac) - f_i(Sa) = \sum_{i=1}^k f_i(Sbc) - f_i(Sb).\]
    Let $v_i = f_i(Sac) - f_i(Sa)$ and $u_i = f_i(Sbc) - f_i(Sb)$. Since $f_i$ are the supermodular functions of a matroid and $|Sac - Sa| = |Sbc - Sb| = 1$, we have that $0 \leq v_i, u_i \leq 1$. Since $M_{i} \twoheadrightarrow M_{i-1}$, we have that $v_{i} \geq v_{i-1}$ and $u_{i} \geq u_{i-1}$. Therefore, the sequences $(v_1,\ldots,v_k)$ and $(u_1,\ldots, u_k)$ consist of a string of $1$s followed by a string of $0$s. Hence, the statement $\sum u_i = \sum v_i$ implies that $u_i = v_i$ for all $i$.

    To show that $f_i$ satisfies the positive tropical Pl\"ucker relation $(S,a,b,c)$, it remains to show that $f_i(Sa) + f_i(Sbc) \leq f_i(Sab) + f_i(Sc)$. In other words, that $f_i(Sbc) - f_i(Sc) \leq f_i(Sab) - f_i(Sa)$. We have
     \[ g(Sbc) - g(Sc) = \sum_{i=1}^k f_i(Sbc) - f_i(Sc) \leq \sum_{i=1}^k f_i(Sab) - f_i(Sa) \leq g(Sab) - g(Sa).\]
    Using the matroid quotient relations as above, the sequences consisting of the terms $f_i(Sbc) - f_i(Sc)$ and the terms $f_i(Sab) - f_i(Sa)$ both consist of a string of $1$s followed by a string of $0$s. Therefore, by the inequality above, we must have that $f_i(Sbc) - f_i(Sc) \leq f_i(Sab) - f_i(Sa)$ for all $i \in [k]$. This implies the wanted inequality.  
\end{proof}

\begin{proof}[Proof of Theorem \ref{mainthm: MV polytopes flag matroids}]
The first part of the statement follows from the previous two propositions and  the second part of comes from the classification of matroid polytopes which are MV polytopes.
\end{proof}

As Kamnitzer showed in \cite{Kamnitzer2007}, it is not true in general that the Minkowski sum of MV polytopes is again an MV polytope. Nevertheless, the set of submodular functions of MV polytopes form a polyhedral subcomplex of the cone of submodular functions. The generators of the cones of this complex are certain sets of MV polytopes called \textbf{clusters}. The structure, enumeration, or properties of the clusters of MV are poorly understood. The proof of the theorem above actually implies the following stronger result:

\begin{theorem}
    Let $(M_1,\ldots,M_k)$ be a flag matroid and let $a_i > 0$. Then the Minkowski sum $\sum_{i=1}^k a_i P(M_i)$ is an MV polytope if and only if each $P(M_i)$ is an MV polytope.
\end{theorem}

This means that concordant matroids appear together in a cluster. However, it is not true that matroids which are concordant are the only matroids whose polytopes can be summed to still obtain MV polytopes. We will later see  sums of Schubert matroid polytopes of the same rank which are still MV polytopes. 

\subsection{Bruhat Interval Polytopes}

 For any $k \in [n]$, define the \textbf{projection map} $\operatorname{proj}_k: S_n \to \binom{[n]}{k}$ by 
\[\operatorname{proj}_k(w) = \{w(1),w(2),\ldots,w(k)\}. \]
The \textbf{Bruhat interval polytope} of the interval $[u,v]$ is the polytope
 \[P_{[u,v]} = \operatorname{conv}((w(1),w(2),\ldots,w(n)) \; \lvert \; u \leq w \leq v). \]

These polytopes were introduced by Kodama and Williams \cite{Kodama2014} and were studied more fully by Tsukerman and Williams \cite{Tsukerman2015}. We will use a different convention, coming from, \cite{boretsky2022polyhedral} which is more convenient for us.

\begin{definition}
    The \textbf{twisted Bruhat interval polytope} $\tilde{P}_{[u,v]}$ is the polytope
      \[\tilde{P}_{[u,v]} = \operatorname{conv}((n+1-w^{-1}(1), n+1-w^{-1}(2), \ldots, n+1-w^{-1}(n)) \; \lvert \; w \in [u,v]). \]
\end{definition}

The set of Bruhat interval polytopes is the same as the set of twisted Bruhat interval polytopes. The difference is only a matter of labeling.

Tsukerman and Williams proved that these polytopes are all flag matroid polytopes. Recall that a \textbf{positroid} is a matroid such that there exists a full-rank $k \times n$ matrix $M$ where every Pl\"ucker coordinate is non-negative and $p_S(M) \not = 0$ if and only if $S$ is a basis of the matroid. These include all lattice path matroids.

\begin{theorem}\cite{Tsukerman2015}\label{thm: williams positroid flag variety}
    Every Bruhat interval polytope is the flag matroid polytope of a full flag matroid whose components are positroids.

    Explicitly,
     \[ \tilde{P}_{[u,v]} = P(M_1) + \cdots + P(M_n),\]
    where the bases of $M_i$ are the subsets in $\operatorname{proj}_i([u,v])$.
\end{theorem}

The converse of this statement is not true since there are flag matroids where every component is a positroid but the corresponding polytope is not a Bruhat interval polytope. However, Benedetii and Knauer \cite{benedettiLPFM} proved that the converse holds when each component is a lattice path matroid.

\begin{theorem}
    The flag matroid polytope of a lattice path flag matroid is a Bruhat interval polytope.
\end{theorem}

Our characterization of MV flag matroid polytopes, gives a characterization of \MV Bruhat interval polytopes. We say that an interval $[u,v]$ has the \textbf{projection property} if $\operatorname{proj}_k([u,v]) = [\operatorname{proj}_k(u),\operatorname{proj}_k(v)]$ for all $k \in [n]$. In other words, the projections of the interval are all intervals in the Gale order.

\begin{theorem}
    A twisted Bruhat interval polytope $\tilde{P}_{[u,v]}$ for an interval $[u,v]$ is an MV polytope if and only if $[u,v]$ has the projection property.
\end{theorem}

\begin{proof}
    This follows from Theorem \ref{mainthm: MV polytopes flag matroids} and Theorem \ref{thm: williams positroid flag variety} since the projection property holds exactly when each constituent matroid is a lattice path matroid.
\end{proof}

It is not known which intervals of $S_n$ satisfy this projection property. One class of examples is when the permutations of the interval are related by the weak order. Recall that the $u \leq_W v$ in the \textbf{(left) weak Bruhat order} of $S_n$ if there exist reduced expressions such that $u = s_{i_1} \cdots s_{i_k}$ and $v = s_{j_1} \cdots s_{j_m} s_{i_1} \cdots s_{i_k}$.

\begin{proposition}\label{prop: weak order BIP}
        Let $[u,v]$ be an interval in the strong Bruhat order of $S_n$ such that $u \leq_W v$ in the weak order. Then $\tilde{P}_{[u,v]}$ is an \MV Bruhat interval polytope.
\end{proposition}

\begin{proof}
    We prove this by induction on the length of the interval $[u,v]$. The base case is where $v = u$ which follows trivially. Suppose $[u,v]$ is an interval with the projection property and $s_i$ is a simple reflection such that $s_iv > v$. We will show that $[u, s_iv]$ also satisfy the projection property which implies the wanted result by induction.

    Let $v[j]$ denote the subset $\operatorname{proj}_j(v) = \{v(1),\ldots,v(j)\}$. We need to show that for all $k$, $[u[k],s_iv[k]] = \{w[k] \; \lvert \; u \leq w \leq s_iv\}$.

    By the subword description of the strong Bruhat order, the elements of $[u, s_iv]$ consist of either elements in $[u,v]$ or elements of the form $w' = s_iw$ where $w$ is in $[u,v]$ and $s_iw > w$. For elements of the second form, we have
     \[s_iw[k] = \begin{cases}
         w[k] & \text{if $i,i+1 \in w[k]$ or $i,i+1 \not \in w[k]$}\\
         w[k]\backslash \{i\} \cup \{i+1\} & \text{if $i \in w[k]$ and $i+1 \not \in w[k]$}.
     \end{cases}. \]
    Note that it cannot happen that $i+1 \in w[k]$ but $i \not \in w[k]$ since $s_iw > w$ and hence $i$ occurs before $i+1$ in the one-line notation of $w$.

    To show that $\{w[k] \; \lvert \; u \leq w \leq s_iv\} \subseteq [u[k],s_iv[k]]$, we need to show that $u[k] \leq w[k] \leq s_iv[k]$ in the Gale order for any $w$. If $w$ is in $[u,v]$ then this follows by the inductive hypothesis. Otherwise, $w'$ is of the form $s_iw$. By the description above, we see $w[k] = s_iw[k]$ or $w[k] \lessdot s_iw[k]$. Therefore, $s_iw[k] \geq u[k]$. If $w[k] = s_iw[k]$, then $s_iw[k] \leq v[k] \leq s_iv[k]$. Otherwise, $i \in w[k]$ and $i+1 \not \in w$ which is also true for $v$. Therefore, $s_iw[k] \less s_iv[k]$ since they are both obtained from $w[k]$ and $v[k]$ by changing the letter $i$ to $i+1$.

    In the other direction, we want to show that $\{w[k] \; \lvert \; u \leq w \leq s_iv\} \supseteq [u[k],s_iv[k]]$. To do this, let $u[k] \leq S\leq s_iv[k]$. If $s_iv[k] = v[k]$ then the result follows as wanted. Otherwise, $s_iv[k] = v[k] \backslash \{i\} \cup \{i+1\}$. Then, as discussed in the proof of Proposition \ref{prop: crystal of lattice path matroids}, the set $S$ is either in $[u[k], s_iv[k]]$ or is of the form $S' \backslash \{i\} \cup \{i+1\}$ where $u[k] \leq S' \leq v_i[k]$.

    In the first case, we are done. In the second case, then by the inductive hypothesis there is a permutation $u \leq w \leq v$ such that $w[k] = S'$. Then, the permutation $s_iw$ is in $[u, s_iv]$ since $s_iw > w$ as $i \in w[k]$ but $i+1 \not \in w[k]$. Further, $s_iw[k] = S$ which proves this direction.
\end{proof}

Let $G/B$ be the flag variety of type $A_n$ and let $\Omega_w$ and $\Omega^w$ be the Schubert and opposite Schubert varieties corresponding to $w$. The \textbf{Richardson variety} $X_u^v$ is the intersection $\Omega_u \cap \Omega^v$ for $u \leq v$. Let $\Phi: G/B \to \operatorname{Lie}(T)^* \cong \RR^n$ be the moment map of $G/B$. Then, $\tilde{P}_{[u,v]}$ is the image of $X_u^v$ under $\Phi$. From this perspective, the \MV Bruhat interval polytopes correspond to the Richardson varieties of $G/B$ such that all the projections onto the components are again Richardson varieties.

\section{Schubitopes}\label{Sec: Schubitope}

We now study Schubitopes which are Newton polytopes of polynomials coming from the Bott-Samelson resolutions of the flag variety \cite{Magyar1998}. We will give a family of these characters whose Newton polytopes is an \MV polytope. This will include the Newton polytopes of Schubert polynomials and key polynomials.

A \textbf{diagram} $\mathcal{D}$ is a finite multset $\{C_1,\ldots,C_k\}$ of subsets of $[n]$. We identify $\mathcal{D}$ with the subsets of boxes in the $k \times n$ grid by $(i,j)$ is in the subset of boxes if and only if $i \in S_j$.

\subsection{Flagged Weyl Modules and Schubitopes}

Given a diagram $\mathcal{D} = \{C_1, \ldots, C_k\}$, let $\{S_1,\ldots,S_m\}$ be the underlying subset and let $a_i$ be the multiplicity of the set $S_i$ in $\mathcal{D}$. Let $\operatorname{Gr}(\mathcal{D})$ be the product of Grassmannians $\operatorname{Gr}(|S_1|,n) \times \cdots \times \operatorname{Gr}(|S_m|,n)$. Let $z_{S_i}$ denote the $T$-fixed point of $\operatorname{Gr}(|S_i|,n)$ where the only non-vanishing Pl\"ucker coordinate is $p_{S_i}$. Consider the point $z_{\mathcal{D}} = (z_{S_1}, \ldots, z_{S_m}) \in \operatorname{Gr}(\mathcal{D})$ and define
 \[\mathcal{F}_{\mathcal{D}}^B = \overline{z_{\mathcal{D}} \cdot B} \subseteq \operatorname{Gr}(\mathcal{D}). \]
Let $\mathcal{L}_{\mathcal{D}}$ be the line bundle of $F_{\mathcal{D}}^B$ given as the pullback of $\mathcal{O}(a_1,a_2,\ldots,a_m)$ on $\operatorname{Gr}(\mathcal{D})$.

The \textbf{dual flagged Weyl module} is the $B$-module $H^{0}(\mathcal{F}_{\mathcal{D}}^B, \mathcal{L}_{\mathcal{D}})^*$. Let $\chi_{\mathcal{D}}$ denote its character viewed as a polynomial in the diagonal entries $x_1,\ldots,x_n$. The main motivation for the study of this module is that its character $\chi_{\mathcal{D}}$ includes Schubert polynomials and key polynomials in special cases.

For a polynomial $f(x) = \sum_{\alpha \in \NN^n} c_{\alpha} x^{\alpha}$ where $x^{\alpha} = x_1^{\alpha_1}\cdots x_n^{\alpha_n}$, the \textbf{Newton polytope} is the polytope
 \[\operatorname{Newton}(f) = \operatorname{conv}(\alpha \in \NN \; \lvert \; c_{\alpha} \not = 0). \]

\begin{definition}
    Let $\mathcal{D}$ be diagram and $\chi_{\mathcal{D}}$ be the character of the corresponding dual flagged Weyl module. The \textbf{Schubitope} $\mathcal{S}_{\mathcal{D}}$ is the polytope $\operatorname{Newton}(\chi_{\mathcal{D}})$. 
\end{definition}

Fink, M\'ez\'aros, and St.\@ Dizier proved in \cite{Fink2018} that every Schubitope is a generalized permutahedron and related them to matroid polytopes.

\begin{theorem}\cite{Fink2018}
    For a subset $A \subseteq [n]$, let $\Omega_A$ be the corresponding Schubert matroid. Let $\mathcal{D} = \{S_1,\ldots,S_k\}$ be the multiset of subsets of $[n]$. Then, the Schubitope $\mathcal{S}_{\mathcal{D}}$ is the Minkowski sum
        \[\mathcal{S}_\mathcal{D} = \sum_{i=1}^k P(\Omega_{S_i}).\]
    In particular, $\mathcal{S}_\mathcal{D}$ is a generalized permutahedron.
\end{theorem}
For any subset $S \subseteq [n]$, define the word $\operatorname{word}(j,S,\mathcal{D})$ by reading the $j$th column of the diagram $\mathcal{D}$ and recording
\begin{enumerate}
    \item ( if $(i,j) \not \in \mathcal{D}$ and $i \in S$;
    \item ) if $(i,j) \in \mathcal{D}$ and $i \not \in S$;
    \item $\star$ if $(i,j) \in \mathcal{D}$ and $i \in S$. 
\end{enumerate}
Then, define $\theta(j,S,\mathcal{D})$ as the number of paired ()'s plus the number of $\star$'s in $\operatorname{word}(j,S,\mathcal{D})$. We will pair parentheses with the usual ``inside-out" convention where the inside pairs are paired first. Fan and Guo used these words to describe the submodular function of $\mathcal{S}_\mathcal{D}$.

\begin{proposition}\cite{Fan2021}\label{prop: submodular function as words}
    The submodular function of $\mathcal{S}_\mathcal{D}$ is given by
     \[\mu(S) = \sum_{j=1}^{k} \theta(j,S,\mathcal{D}). \]
\end{proposition}

\begin{example}\label{ex: computation of theta(123)}
Below is the diagram given by the multiset $\{13,24,24,35\}$ of subsets of $[n]$ alongside a labeling of the the boxes according to the rules above to compute $\mu(123)$.

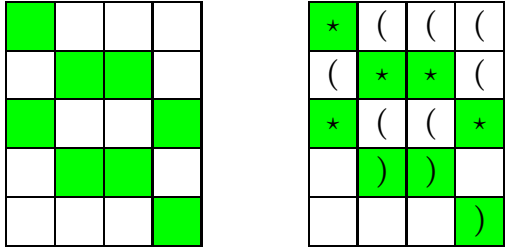
\begin{center}
    \begin{ytableau}
        *(green) & & & \\
        & *(green) & *(green) & \\
        *(green)& & & *(green) \\
        & *(green)& *(green)& \\
        & & &*(green)\\
    \end{ytableau}
    \quad \quad \quad
         \begin{ytableau}
        *(green) \star & ( & (& (\\
        ( & *(green) \star & *(green)\star & (\\
        *(green) \star & (&( & *(green)\star \\
        & *(green) ) & *(green) )& \\
        & & &*(green) )\\
    \end{ytableau}
    \end{center}

To compute $\mu(123)$, the words are read off column by column from this labeled diagram. We see three pairs of paired parentheses and five stars which gives $\mu(123) = 8$.
\end{example}
\subsection{Crystal Structure in terms of Demazure Operators}

The \textbf{Demazure operator} (or isobaric divided difference operator) $\Lambda_i : \CC[x_1,\ldots,x_n] \to \CC[x_1,\ldots,x_n]$ is the operator
 \[ \Lambda_i f = \frac{x_if - x_{i+1}s_if}{x_i - x_{i+1}},\]
where the simple reflection $s_i$ acts by swapping the variables $x_i$ and $x_{i+1}$. 

We say that a diagram $\mathcal{D}$ has an \textbf{ascent} at $i$ if each column $C_k$ satisfies $C_k \cap \{i, i+1\} \not = \{i+1\}$. Let $s_i\mathcal{D}$ be the diagram obtained by swapping rows $i$ and $i+1$ in the diagram $\mathcal{D}$. When $\mathcal{D}$ has an ascent at $i$, swapping row $i$ with $i+1$ corresponds to applying the $i$th Demazure operator.

\begin{theorem}\cite{Magyar1998}
    Let $\mathcal{D}$ be a diagram with an ascent at $i$. Then,
     \[\chi_{s_i\mathcal{D}} = \Lambda_i \chi_\mathcal{D}. \]
\end{theorem}

The key to studying Schubitopes is that the raising operators of the crystal of \MV polytopes is closely related to the Demazure operator.

\begin{proposition}
    Let $\mathcal{D}$ be a diagram with an ascent at $i$. Then, the submodular function $\mu'$ of the polytope $\mathcal{S}_{s_i\mathcal{D}}$ is related to the submodular function $\mu$ of the polytope $\mathcal{S}_{\mathcal{D}}$ by
     \[ \mu'(S) = \begin{cases}
        \mu(S \cup \{i\} \backslash \{i+1\}) & \text{ if $i \not \in S$ and $i +1 \in S$} \\
        \mu (S) & \text{otherwise.}
     \end{cases}\]
\end{proposition}

\begin{proof}
    We will prove this by comparing $\operatorname{word}(j,S,\mathcal{D})$ with the $\operatorname{word}(j,S,s_i\mathcal{D})$ which by Proposition \ref{prop: submodular function as words} will relate the two submodular functions. Each column is independent and so it suffices to compare the submodular function column by column. Further, if there is a column $C$ where $i$ and $i+1$ are either both in $C$ or both not in $C$, then the word of that column does not change after the swap (as $C$ itself does not change). The only remaining type of column $C$ is where $i \in C$ and $i+1 \not \in C$ since $\mathcal{D}$ has an ascent at $i$.

    In what follows, we draw figures as in Example \ref{ex: computation of theta(123)}. 

\vskip 2ex
    \noindent \textbf{Case 1:} $i \in S$ and $i+1 \in S$
    
    In this case, the letters of the word corresponding to the the boxes in rows $i$ and $i+1$ of $\mathcal{D}$ and $s_i\mathcal{D}$ in a column are related by
    \ytableausetup{centertableaux}
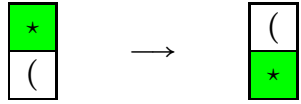
    \begin{center}
    \begin{ytableau}
        *(green)\star \\
        (
    \end{ytableau}
    $\quad\quad \longrightarrow\quad\quad$ 
    \begin{ytableau}
        ( \\
        *(green)\star
    \end{ytableau}
    \end{center}
    It is clear that this operation does not change the number of stars or the number of paired parentheses in the full diagram. Hence, $\mu'(S) = \mu(S)$.

    \vskip 2ex
    \noindent \textbf{Case 2:} $i \in S$ and $i+1 \not \in S$

    In this case, the letters of rows $i$ and $i+1$ of a column of $\mathcal{D}$ and $s_i\mathcal{D}$ are related by
     
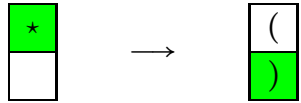
\begin{center}
    \begin{ytableau}
        *(green)\star \\
        
    \end{ytableau}
    $\quad\quad \longrightarrow\quad\quad$ 
    \begin{ytableau}
        ( \\
        *(green) )
    \end{ytableau}
    \end{center}
    By pairing these two adjacent parentheses together, we see that the number of stars decreases by one and the number of paired parenthesis increases by one. Hence $\mu'(S) = \mu(S)$.

     \vskip 2ex
    \noindent \textbf{Case 3:} $i \not \in S$ and $i+1 \not \in S$
    The configurations are related by
   
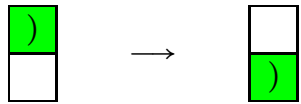
\begin{center}
    \begin{ytableau}
        *(green) )\\
        
    \end{ytableau}
    $\quad\quad \longrightarrow\quad\quad$ 
    \begin{ytableau}
         \\
        *(green) )
    \end{ytableau}
    \end{center}
    It is clear that the number of stars and paired parentheses does not change again. So, $\mu'(S) = \mu(S)$.

    \vskip 2ex
    \noindent \textbf{Case 4:} $i \not \in S$ and $i+1 \in S$

    The local configuration changes by
     
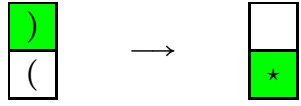
\begin{center}
    \begin{ytableau}
        *(green) )\\
        (
    \end{ytableau}
    $\quad\quad \longrightarrow\quad\quad$ 
    \begin{ytableau}
         \\
        *(green) \star
    \end{ytableau}
    \end{center}

    Compare this to the word of the column of $\mathcal{D}$ corresponding to the subset $S \backslash \{i+1\} \cup \{i\}$ given by

    
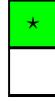
\begin{center}
            \begin{ytableau}
        *(green) \star\\
        
    \end{ytableau}
    \end{center}
  From this, it is clear that $\mu'(S) = \mu(S \backslash \{i+1\} \cup \{i\})$ since the column above and the column of $s_i\mathcal{D}$ have the same word.

Summing the contributions of each column to $\mu(S)$ we obtain the wanted calculation.

\end{proof}

We compare the change in the submodular function caused by swapping rows to the application of the crystal operators.

\begin{proposition}
    Let $\mathcal{D}$ be a family of subsets of $[n]$ that has an ascent at $i$. Then, the submodular function $\mu'$ of $e_i^k \cdot \mathcal{S}_\mathcal{D}$ is related to the submodular function $\mu$ of $\mathcal{S}_{\mathcal{D}}$ by
    \[ \mu'(S) = \begin{cases}
        \max(\mu(S), \mu(S \cup \{i\} \backslash \{i+1\} + (\ell-k)) & \text{ if $i \not \in S$ and $i +1 \in S$} \\
        \mu (S) & \text{otherwise,}
     \end{cases}\]
    where $\ell$ is the number of subsets $C \in \mathcal{D}$ such that $i \in C$ and $i+1 \not \in C$.
\end{proposition}

\begin{proof}
    Let $\mu$ be the submodular function of $\mathcal{S}_{\mathcal{D}}$ and $\mu_k$ be the submodular function of $\mathbf{e}_i^k \mathcal{S}_{\mathcal{D}}$.
    
    We proceed by induction on $k$. The base case is when $k = 1$. By Proposition \ref{prop: submodular function of crystal action}, the submodular function of $\mu_1$ is
    \[\mu_1(S) = \begin{cases}
        \max( \mu(S), \mu(S \backslash \{i+1\} \cup \{i\}) - c) & \text{if $i \not \in S$ and $i+1 \in S$} \\
         \mu(S) & \text{otherwise,} \\
     \end{cases} \]
    where $c = \mu([n] \backslash s_i[i]) - \mu([n] \backslash [i]) - 1$. This value $c$ can be calculated directly. Note that if $C$ is a subset of $\mathcal{D}$ with $i$ and $i+1 \in C$ or $i$ and $i+1 \not \in C$, then we have that the contribution of this subset to $\mu([n] \backslash s_i[i]) - \mu([n] \backslash [i])$ is $0$. If $C$ is a column with $i \in C$ and $i + 1 \not \in C$, then the contribution of that set to the sum is $1$. Since these are the only possibilities (as $\mathcal{D}$ has an ascent at $i$), we have that $c = \ell - 1$.

    For the inductive step, suppose this calculation holds for $k$ and we prove it for $k+1$. We first use Proposition \ref{prop: submodular function of crystal action}, to express $\mu_{k+1}$ in terms of $\mu_k$. We have
          \[\mu_{k+1}(S) = \begin{cases}
        \max( \mu_k(S), \mu_k(S \backslash \{i+1\} \cup \{i\}) - c) & \text{if $i \not \in S$ and $i+1 \in S$} \\
         \mu_k(S) & \text{otherwise,} \\
     \end{cases} \]
    where $c = \mu_k([n] \backslash s_i[i]) - \mu_k([n] \backslash [i]) - 1$. Using the inductive hypothesis, we can simplify the expression above to
     \[ \mu_{k+1}(S) =  \max \left ( \mu(S), \mu(S \backslash (i+1) \cup i) - (\ell - k), \mu(S \backslash (i+1) \cup i) - c \right ) \]
     when $i \not \in S$ and $i+1 \in S$, and $\mu_{k+1}(S) = \mu(S)$ otherwise. 
     We can also simplify the expression for $c$ with the inductive hypothesis by
    \begin{align*}
        c &= \mu_k([n] \backslash s_i[i]) - \mu_k([n] \backslash [i]) - 1 \\
        &= \mu([n] \backslash s_i[i]) - \max \left ( \mu([n] \backslash [i]), \mu([n] \backslash s_i[i]) - (\ell - k) \right ) - 1 \\
        &= \min \left ( \mu([n] \backslash s_i[i]) - \mu([n] \backslash [i]), \ell - k \right ) -1 \\
        &= \min(\ell, \ell - k) - 1 \\
        &= \ell - k - 1
    \end{align*}
    where the second-to-last equality comes from the arguments of the base case. Using this expression for $c$ in the expression of $\mu_{k+1}(S)$, we obtain the wanted result.
\end{proof}

Combining these two propositions, we obtain the wanted relationship between Demazure operators and the raising operators of \MV polytopes.

\begin{theorem}\label{thm: crystal structure demazure operator}
    Let $\mathcal{D}$ be a diagram that has an ascent at $i$. Let $\ell$ be the number of subsets $C$ of $\mathcal{D}$ such that $i \in C$ and $i+1 \not \in C$. Then,

    \[\mathbf{e}_i^{\ell} \cdot \mathcal{S}_{\mathcal{D}} = \mathcal{S}_{s_i \mathcal{D}} = \operatorname{Newton}(\Lambda_i  \chi_\mathcal{D}). \]
\end{theorem}

\begin{proof}
    This follows from the previous two propositions with the minor observation that $\mu(S\cup \{i\} \cup \{i+1\}) \geq \mu(S)$ for any $S$ with $i \not \in S$, $i+1 \in S$ since $\mathcal{D}$ has an ascent at $i$.
\end{proof}

\subsection{Separated Families}

We now use this relationship to give a large class of Schubitopes which are \MV polytopes. We say that a diagram $\mathcal{D}$ is \textbf{strongly separated} if for any two $C_1, C_2 \in \mathcal{D}$ we have
 \[(C_1 \backslash C_2) \leq_{elt} (C_2 \backslash C_1) \quad \quad \text{or} \quad \quad (C_2 \backslash C_1) \leq_{elt} (C_1 \backslash C_2),\]
where $A \leq_{elt}B$ if every element of $A$ is smaller than every element of $B$. The diagram $\mathcal{D}$ being strongly separated is equivalent to it being a \%-avoiding diagram, as in \cite{Reiner1998}, when the columns of $\mathcal{D}$ are organized in increasing lexicographical order.

Reiner and Shimozono \cite{Reiner1998} define the \textbf{orthodontic} partial order on the set of diagrams where the covering relations are defined by the following cases:

\begin{itemize}
    \item (Or1) If the lexicographical minimal subset $C_i \in \mathcal{D}$ is an interval $[k]$ for some $k$, then
     \[\mathcal{D} \backslash C_i \lessdot \mathcal{D}.\]
    \item (Or2) If $\mathcal{D}$ has an ascent at $i$, then
     \[ \mathcal{D} \lessdot s_i\mathcal{D}.\]
\end{itemize}

They use this to give a characterization of separated families.
\begin{theorem}\label{thm: separated means orthodontic}\cite{Reiner1998}
    The diagram $\mathcal{D}$ is strongly separated if and only if $\emptyset \leq \mathcal{D}$ in the orthodontic order, where $\emptyset$ is the empty diagram.
\end{theorem}

We can now prove our main theorem of this section.

\begin{theorem}[Theorem \ref{mainthm: MV polytopes Schubitopes}]
    Let $\mathcal{D}$ be a strongly separated diagram. Then, the Schubitope $\mathcal{S}_D$ is an MV polytope.
\end{theorem}

\begin{proof}

    By Theorem \ref{thm: separated means orthodontic}, there is a sequence of diagrams
     \[\emptyset = \mathcal{D}_0 \lessdot \mathcal{D}_1 \lessdot \cdots \lessdot \mathcal{D}_m = \mathcal{D} \]
    where at each step we either add a new column of the form $[k]$ or we apply $s_i$ to a diagram that has an ascent at $i$. In the first case, the Schubitope $\mathcal{S}_{\mathcal{D}_j}$ changes by a translation since the Schubert matroid of an initial interval $[k]$ has a unique basis, and so its matroid polytope is a point. In the second case, our Schubitope changes by an application of raising operator by Theorem \ref{thm: crystal structure demazure operator}.

    Both of these operations preserve \MV polytopes and the point $\mathcal{S}_{\emptyset}$ is an \MV polytopes. Hence, $\mathcal{S}_{\mathcal{D}}$ is an \MV polytope.
\end{proof}

\begin{remark}
    The converse of this theorem is not true. Indeed the Schubitope corresponding to the the diagram $\{14, 23\}$ is an \MV polytope even though it is not strongly separated. Not all Schubitopes are \MV polytopes. For instance, the diagram $\{13,2\}$ does not give an \MV polytope.
\end{remark}

\subsection{Newton Polytopes of Symmetric Functions}

One of the main motivation behind Schubitopes was to study Newton polytopes of various important symmetric functions. 

Define the \textbf{divided difference operator} on  polynomials in $n$ variables is given by
 \[ \partial_i f = \frac{f - s_if}{x_i - x_{i+1}}.\]
For any reduced word $w = s_{i_1} \cdots s_{i_{\ell}}$, let $\partial_w = \partial_{i_1} \cdots \partial_{i_{\ell}}$. This operator is independent of the chose of reduced word for $w$. Let $w_0$ be the longest permutation in $S_n$. For any $w$, let $u = w^{-1}w_0$. Then, the \textbf{Schubert polynomial} $X(w)$ is the polynomial 
 \[X(w) = \partial_u(x_1^{n-1}x_2^{n-2} \cdots x_{n-1}). \]

Recall that the \textbf{Rothe diagram} of a permutation $w$ is the diagram
 \[D(w) = \{ (i,j) \subseteq [n] \times [n] \; \lvert \;  i < w^{-1}(j), j \leq w(i)\}. \]

It follows by work of Kra\'skiewicz and Pragacz \cite{KP} that the character $\chi_{D(w)}$ is the Schubert polynomial $X(w)$. The Rothe diagram of a permutation is strongly separated, see for instance \cite{Reiner1995}. Therefore, we have the following consequence of Theorem \ref{mainthm: MV polytopes Schubitopes}.
 
\begin{corollary}
    The Newton polytope of a Schubert polynomial is an MV polytope.
\end{corollary}

The \textbf{key polynomial} $\kappa_{\alpha}$ for a composition $\alpha$ of $n$ is defined recursively by

\begin{enumerate}
    \item If $\alpha$ is weakly decreasing, then 
     \[\kappa_{\alpha} = x_1^{\alpha_1} \cdots x_k^{\alpha_k}. \]
     \item Otherwise, let $i$ be an entry such that $\alpha_i < \alpha_{i+1}$ and let $\alpha' = s_i(\alpha)$ be the composition obtained by switching $\alpha_i$ and $\alpha_{i+1}$. Then,
      \[\kappa_{\alpha} = \partial_i(x_i \kappa_{\alpha'}). \]
\end{enumerate}

These polynomials were introduced by Demazure \cite{demazure1974nouvelle} and studied combinatorially by Lascoux and Sch\"utzenberger \cite{Lascoux1989}. Given a composition $\alpha = (\alpha_1, \ldots, \alpha_n)$, the \textbf{skyline diagram} $D(\alpha)$ is the diagram where row $i$ contains the boxes $(1,i),\ldots,(\alpha_i,i)$. Then, $\chi_{D(\alpha)} = \kappa_{\alpha}$. As a corollary of Theorem \ref{mainthm: MV polytopes Schubitopes}, we have

\begin{corollary}
    The Newton polytope of a key polynomial is an MV polytope.
\end{corollary}

\bibliographystyle{plain}
\bibliography{ref.bib}
\end{document}